\pdfoutput=1
\documentclass[11pt]{article}

\usepackage[margin=1in]{geometry}
\usepackage[T1]{fontenc}
\usepackage[utf8]{inputenc}
\usepackage{textcomp}
\usepackage{amsmath,amssymb,amsthm,mathtools,mathrsfs}
\usepackage{microtype}
\usepackage{enumitem}
\usepackage{booktabs}
\usepackage{array}
\usepackage{hyperref}

\hypersetup{
  colorlinks=true,
  linkcolor=blue,
  citecolor=blue,
  urlcolor=blue
}

\newtheorem{theorem}{Theorem}[section]
\newtheorem{proposition}[theorem]{Proposition}
\newtheorem{definition}[theorem]{Definition}
\newtheorem{lemma}[theorem]{Lemma}
\newtheorem{remark}[theorem]{Remark}
\newtheorem{assumption}[theorem]{Assumption}
\newtheorem{corollary}[theorem]{Corollary}
\newtheorem{example}[theorem]{Example}

\newcommand{\R}{\mathbb{R}}
\newcommand{\Om}{\Omega}
\newcommand{\pa}{\partial}
\newcommand{\cl}{\overline}
\newcommand{\grad}{\nabla}
\newcommand{\diver}{\operatorname{div}}

\newcommand{\id}{\mathrm{id}}
\newcommand{\dd}{\,\mathrm{d}}

\newcommand{\F}{\mathcal{F}}
\newcommand{\X}{\mathcal{X}}
\newcommand{\Diss}{\mathcal{D}}
\newcommand{\Bnd}{\mathcal{B}}
\newcommand{\Src}{\mathcal{S}}

\newcommand{\ThermCont}{\mathbf{ThermCont}}
\newcommand{\ThermAdm}{\mathbf{ThermCont}^{\mathrm{adm}}}
\newcommand{\ThermLin}{\mathbf{ThermLin}}

\newcommand{\lin}{\mathrm{lin}}
\newcommand{\pair}[2]{\left\langle #1,#2\right\rangle}

\begin{document}

\title{Thermodynamic Structure and Composition in Nonlinear Convection--Diffusion}

\author{J.\ J.\ Segura}
\date{}

\maketitle

\noindent\textbf{Publication note.}
This manuscript is the author-prepared arXiv version of the article published in \textit{Open Transport}. The Version of Record is available at DOI: \href{https://doi.org/10.1515/ot-2026-0013}{10.1515/ot-2026-0013}.

\medskip

\begin{abstract}
Nonlinear convection--diffusion systems play a central role in transport phenomena,
including mass transfer, heat transfer, porous-media transport, and coupled continuum
processes with source, exchange, and interface effects. In such systems, the key
question is often not only which governing partial differential equation is used, but
whether the model preserves a consistent thermodynamic balance under the operations
that arise naturally in transport analysis: restriction to subdomains, coupling
across interfaces, linearization near equilibrium, and discretization for
computation.

This paper develops a continuum-first framework for open nonlinear
convection--diffusion systems in which thermodynamic consistency is formulated as a
free-energy balance with nonnegative bulk dissipation and explicit boundary and source
contributions. Within this setting, nonlinear transport systems are defined as
structured objects built from admissible state fields, storage functionals,
constitutive flux decompositions, sources, and boundary ports. We prove that the
thermodynamic balance is preserved under exact structure-preserving transformations,
restriction to subdomains, local-to-global reconstruction over compatible domain
decompositions, and power-conserving interconnection of open subsystems. We then
derive classical linear convection--diffusion models as tangent thermodynamic
descendants at equilibrium and show that the same invariant survives weak
formulation, semidiscretization, and fully discrete time stepping when the numerical
design respects thermodynamic structure. Nonlinear drift--diffusion and porous-medium
convection--diffusion are used as explicit examples. The resulting contribution is a
compositional transport framework in which the second law remains visible across
continuum modeling, subsystem coupling, linear approximation, and computation.
\end{abstract}



\section{Introduction}

Nonlinear convection--diffusion systems lie at the core of transport phenomena. They
govern the coupled movement of mass, heat, and generalized scalar quantities in
fluids, porous media, reactive continua, and multiscale engineering systems. In such
problems, convection, diffusion, source terms, and boundary exchange do not merely
coexist; they interact in ways that strongly affect admissibility, stability, and
physical interpretation. For this reason, a satisfactory transport model should do
more than reproduce a governing differential equation. It should also preserve the
relevant thermodynamic balance, especially when the system is restricted to a
subdomain, coupled to another subsystem across an interface, linearized near an
equilibrium state, or discretized for computation.

Several existing traditions address parts of this structural problem. Classical
nonequilibrium thermodynamics emphasizes constitutive consistency and entropy
production \cite{OnsagerI,OnsagerII,deGrootMazur}; gradient-flow and
energy--dissipation approaches clarify nonlinear relaxation and free-energy decay
\cite{JKO,OttoPME,Mielke2011,GlitzkyMielke2013}; and open-system or
boundary-energy formulations make exchange with the environment mathematically
explicit \cite{vanDerSchaftMaschke,GrmelaOttingerI,GrmelaOttingerII}. The present
paper takes a complementary route tailored to nonlinear continuum transport. We treat
open convection--diffusion systems as primitive thermodynamic objects and study which
natural transport operations preserve thermodynamic consistency. This leads to a
categorical formulation, but the motivation remains physical: the goal is to organize
nonlinear transport models so that the second law survives restriction,
interconnection, linearization, and discretization in a unified way.

The point is not that every transport PDE must be rewritten categorically for its own
sake. A single equation can certainly be studied without categorical language. The
need for such language appears when one asks whether the second-law structure survives
the operations that practitioners actually perform on models:
\begin{itemize}[leftmargin=1.2cm]
    \item restriction to subdomains;
    \item interconnection across interfaces;
    \item constitutive reparameterization;
    \item linearization near equilibrium;
    \item semidiscretization and time integration.
\end{itemize}
These are not questions about one PDE in isolation. They are questions about a class
of structured objects and admissible transformations between them.

The paper is therefore organized around a reversal of the usual order of exposition.
We begin with nonlinear open continuum systems carrying a free-energy balance. Only
later do we derive linear convection--diffusion theory as a tangent theory near
equilibrium. In this sense, the framework is \emph{continuum-first} and
\emph{nonlinear-first}; linear theory is retained, but demoted from foundation to
corollary.

\paragraph{Main contributions.}
The manuscript establishes the following program.
\begin{enumerate}[label=(\roman*),leftmargin=1.1cm]
    \item It defines a class of nonlinear open continuum thermodynamic systems built
    from a domain, an admissible field space, a storage density, constitutive fluxes,
    sources, and boundary ports.
    \item It proves a continuum free-energy balance from constitutive admissibility.
    \item It defines a category of such systems and proves invariance of
    thermodynamic consistency under exact morphisms.
    \item It proves closure of thermodynamic consistency under restriction,
    local-to-global reconstruction, and compatible interconnection.
    \item It derives linear convection--diffusion systems as tangent thermodynamic
    systems at reference equilibria.
    \item It shows that the same invariant survives weak formulation,
    semidiscretization, and fully discrete time stepping under structure-preserving
    design conditions.
    \item It works out two explicit PDE classes: nonlinear drift--diffusion and
    porous-medium convection--diffusion.
\end{enumerate}

\paragraph{Theorem roadmap.}
The mathematical development proceeds in four layers. First, we define nonlinear
open continuum thermodynamic systems and derive their free-energy balance from
constitutive admissibility. Second, we organize these systems into a category and
show that thermodynamic consistency is preserved under exact morphisms, restriction,
local-to-global reconstruction, and interconnection. Third, we recover linear
convection--diffusion theory as the tangent thermodynamic descendant at equilibrium.
Fourth, we show that the same invariant survives weak formulation and
structure-preserving discretization.

\paragraph{Why a categorical formulation is genuinely useful.}
A single convection--diffusion equation can, of course, be studied without
categorical language. The need for such language arises only when one asks whether
the second-law structure survives the operations that are actually performed on
continuum models: restriction to subdomains, interconnection across interfaces,
constitutive transformation, linearization, and discretization. These are not
questions about one PDE in isolation. They are questions about a class of structured
objects and admissible transformations between them. The categorical layer is
therefore not decorative. Its role is to make precise the claim that thermodynamic
consistency should be preserved under the natural constructions of continuum
transport theory.

\section{Nonlinear continuum thermodynamic systems}

\subsection{Primitive data}

We work first on a bounded Lipschitz domain \(\Om\subset\R^d\) with outward unit
normal \(n\), and a nonempty open convex state set \(U\subset\R^m\).

\begin{definition}[Admissible state space]
Let \(\psi:\cl{\Om}\times U\to\R\) be a \(C^2\) storage density. The admissible field
space is
\[
\X
:=
\left\{
u\in H^1(\Om;\R^m):
u(x)\in U \text{ for a.e. }x\in\Om,\;
\psi(\cdot,\nu(\cdot))\in L^1(\Om)
\right\}.
\]
The associated free-energy functional is
\begin{equation}\label{eq:F-def}
\F[\nu]
:=
\int_\Om \psi(x,\nu(x))\,\dd x.
\end{equation}
The thermodynamic potential is
\begin{equation}\label{eq:mu-def}
\mu[\nu](x):=D_u\psi(x,\nu(x)).
\end{equation}
\end{definition}

\begin{assumption}[Convex storage]\label{ass:convex-storage}
For each \(x\in\cl{\Om}\), the map \(u\mapsto \psi(x,u)\) is convex on \(U\).
\end{assumption}

We consider nonlinear balance laws of the form
\begin{equation}\label{eq:master-pde}
\partial_t u + \diver J[u] = R[u]
\qquad \text{in }\Om\times(0,T),
\end{equation}
where the total flux is split as
\[
J[u]=J^{\mathrm{rev}}[u]+J^{\mathrm{diss}}[u].
\]

\begin{definition}[Open constitutive structure]
An \emph{open constitutive structure} on \((\Om,U,\X,\psi)\) consists of:
\begin{enumerate}[label=(\roman*),leftmargin=1.1cm]
    \item a reversible flux \(J^{\mathrm{rev}}:\X\to L^1(\Om;\R^{m\times d})\);
    \item a dissipative flux \(J^{\mathrm{diss}}:\X\to L^1(\Om;\R^{m\times d})\);
    \item a source \(R:\X\to L^1(\Om;\R^m)\);
    \item boundary port spaces \(Y_e,Y_f\) and trace maps
    \[
    e:\X\to Y_e,\qquad f:\X\to Y_f,
    \]
    together with a continuous duality pairing
    \[
    \pair{\cdot}{\cdot}_Y:Y_e\times Y_f\to\R.
    \]
\end{enumerate}
The induced boundary exchange is
\begin{equation}\label{eq:Bnd-def}
\Bnd[u]:=\pair{e[u]}{f[u]}_Y.
\end{equation}
\end{definition}

\begin{definition}[Constitutive thermodynamic admissibility]\label{def:constitutive-adm}
The constitutive structure is called \emph{thermodynamically admissible} if there
exist maps
\[
q^{\mathrm{rev}},q^{\mathrm{diss}}:\X\to L^1(\Om;\R^d),
\qquad
\sigma:\X\to L^1(\Om),
\]
such that for every sufficiently smooth \(u\in\X\),
\begin{align}
\mu[u]\cdot \diver J^{\mathrm{rev}}[u]
&=
\diver q^{\mathrm{rev}}[u]
\qquad \text{in }\mathscr{D}'(\Om),\label{eq:rev-id}\\
\mu[u]\cdot \diver J^{\mathrm{diss}}[u]
&=
\diver q^{\mathrm{diss}}[u]+\sigma[u]
\qquad \text{in }\mathscr{D}'(\Om),\label{eq:diss-id}\\
\sigma[u](x)&\ge 0
\qquad \text{for a.e. }x\in\Om,\label{eq:sigma-pos}\\
-\int_{\pa\Om}\bigl(q^{\mathrm{rev}}[u]+q^{\mathrm{diss}}[u]\bigr)\cdot n\,\dd S
&=
\Bnd[u],\label{eq:boundary-match}
\end{align}
and the source contribution
\begin{equation}\label{eq:source-contrib}
\Src[u]
:=
\int_\Om \mu[u]\cdot R[u]\,\dd x
\end{equation}
is finite.
\end{definition}

\begin{definition}[Open continuum thermodynamic system]
An \emph{open continuum thermodynamic system} is a tuple
\[
\mathsf{T}
=
(\Om,U,\X,\psi,\mu,J^{\mathrm{rev}},J^{\mathrm{diss}},R,Y_e,Y_f,e,f)
\]
satisfying the preceding hypotheses and Definition~\ref{def:constitutive-adm}.
\end{definition}

\begin{remark}
This object should be viewed as a structured continuum system rather than merely as a
PDE, since its thermodynamic meaning depends essentially on storage, constitutive
decomposition, sources, and port variables.
\end{remark}

\subsection{Continuum balance}

\begin{definition}[Smooth admissible trajectory]
A \emph{smooth admissible trajectory} for \(\mathsf{T}\) on \([0,T]\) is a map
\[
u:[0,T]\to \X
\]
that is sufficiently regular in space and time for all terms in
\eqref{eq:master-pde}--\eqref{eq:source-contrib} to be meaningful, and satisfies
\eqref{eq:master-pde} almost everywhere.
\end{definition}

\begin{lemma}[Local thermodynamic balance]\label{lem:local-balance}
Let \(u\) be a smooth admissible trajectory and define
\[
q[u]:=q^{\mathrm{rev}}[u]+q^{\mathrm{diss}}[u].
\]
Then
\begin{equation}\label{eq:local-balance}
\partial_t \psi(\cdot,u) + \diver q[u]
=
-\sigma[u] + \mu[u]\cdot R[u]
\qquad \text{in }\mathscr{D}'(\Om).
\end{equation}
\end{lemma}

\begin{proof}
By the chain rule,
\[
\partial_t \psi(\cdot,u)=\mu[u]\cdot \partial_t u.
\]
Using the evolution equation \eqref{eq:master-pde},
\[
\partial_t u=-\diver J^{\mathrm{rev}}[u]-\diver J^{\mathrm{diss}}[u]+R[u].
\]
Therefore
\[
\partial_t \psi(\cdot,u)
=
-\mu[u]\cdot \diver J^{\mathrm{rev}}[u]
-\mu[u]\cdot \diver J^{\mathrm{diss}}[u]
+\mu[u]\cdot R[u].
\]
Substituting \eqref{eq:rev-id} and \eqref{eq:diss-id} yields
\[
\partial_t \psi(\cdot,u)
=
-\diver q^{\mathrm{rev}}[u]
-\diver q^{\mathrm{diss}}[u]
-\sigma[u]
+\mu[u]\cdot R[u],
\]
which is exactly \eqref{eq:local-balance}.
\end{proof}

\begin{theorem}[Continuum free-energy balance]\label{thm:continuum-balance}
Let \(u\) be a smooth admissible trajectory of \(\mathsf{T}\). Define
\begin{equation}\label{eq:Diss-def}
\Diss[u]
:=
\int_\Om \sigma[u]\,\dd x.
\end{equation}
Then
\begin{equation}\label{eq:continuum-balance}
\frac{\dd}{\dd t}\F[u(t)]
=
-\Diss[u(t)] + \Bnd[u(t)] + \Src[u(t)]
\qquad \text{for all }t\in[0,T].
\end{equation}
In particular, \(\Diss[u(t)]\ge 0\).
\end{theorem}

\begin{proof}
Integrate \eqref{eq:local-balance} over \(\Om\). The left-hand side gives
\[
\frac{\dd}{\dd t}\int_\Om \psi(x,u(x,t))\,\dd x
+
\int_\Om \diver q[u(t)]\,\dd x.
\]
By the divergence theorem and \eqref{eq:boundary-match},
\[
\int_\Om \diver q[u(t)]\,\dd x
=
\int_{\pa\Om} q[u(t)]\cdot n\,\dd S
=
-\Bnd[u(t)].
\]
The right-hand side gives
\[
-\int_\Om \sigma[u(t)]\,\dd x
+
\int_\Om \mu[u(t)]\cdot R[u(t)]\,\dd x
=
-\Diss[u(t)] + \Src[u(t)].
\]
Rearranging yields \eqref{eq:continuum-balance}. Nonnegativity of \(\Diss\) follows
from \eqref{eq:sigma-pos}.
\end{proof}

\begin{definition}[Thermodynamic consistency]
An open continuum thermodynamic system is \emph{thermodynamically consistent} if every
smooth admissible trajectory satisfies \eqref{eq:continuum-balance}.
\end{definition}

\section{The category of open thermodynamic field systems}

\begin{definition}[Objects]
The objects of \(\ThermCont\) are open continuum thermodynamic systems
\[
\mathsf{T}
=
(\Om,U,\X,\psi,\mu,J^{\mathrm{rev}},J^{\mathrm{diss}},R,Y_e,Y_f,e,f)
\]
as defined above.
\end{definition}

The category introduced here is intentionally strict. Its morphisms preserve the
thermodynamic bookkeeping exactly, and for that reason they isolate the invariant
content of the theory with maximal transparency. This exact level is the right
starting point: it tells us what it means for storage, dissipation, source structure,
and boundary exchange to be preserved without ambiguity. More flexible notions of
thermodynamic morphism, allowing controlled defect terms or order-theoretic
weakening, belong to a later extension of the framework rather than to its
foundational layer.

\begin{definition}[Exact thermodynamic morphism]\label{def:exact-morphism}
Let \(\mathsf{T}_1,\mathsf{T}_2\in\ThermCont\). An \emph{exact thermodynamic morphism}
\[
\Phi:\mathsf{T}_1\to \mathsf{T}_2
\]
consists of maps
\[
T_\Phi:\X_1\to \X_2,
\qquad
P^e_\Phi:Y_{e,1}\to Y_{e,2},
\qquad
P^f_\Phi:Y_{f,1}\to Y_{f,2},
\]
such that for every smooth admissible trajectory \(u\) of \(\mathsf{T}_1\), the
transformed trajectory \(u^\Phi:=T_\Phi u\) is a smooth admissible trajectory of
\(\mathsf{T}_2\), and
\begin{align}
\F_2[u^\Phi(t)]&=\F_1[u(t)],\label{eq:morph-F}\\
\Diss_2[u^\Phi(t)]&=\Diss_1[u(t)],\label{eq:morph-D}\\
\Src_2[u^\Phi(t)]&=\Src_1[u(t)],\label{eq:morph-S}\\
e_2[u^\Phi(t)]&=P^e_\Phi e_1[u(t)],\label{eq:morph-e}\\
f_2[u^\Phi(t)]&=P^f_\Phi f_1[u(t)],\label{eq:morph-f}
\end{align}
and the pairings are compatible in the sense that
\[
\pair{P^e_\Phi \eta}{P^f_\Phi \zeta}_{Y_2}
=
\pair{\eta}{\zeta}_{Y_1}
\]
for all \((\eta,\zeta)\in Y_{e,1}\times Y_{f,1}\).
\end{definition}

\begin{proposition}\label{prop:category}
Exact thermodynamic morphisms define a category \(\ThermCont\).
\end{proposition}

\begin{proof}
For each object \(\mathsf{T}\), the identity maps on \(\X\), \(Y_e\), and \(Y_f\)
define an exact morphism \(\id_\mathsf{T}\). If
\[
\Phi:\mathsf{T}_1\to \mathsf{T}_2,
\qquad
\Psi:\mathsf{T}_2\to \mathsf{T}_3
\]
are exact, then
\[
\Psi\circ\Phi
:=
(T_\Psi\circ T_\Phi,\; P^e_\Psi\circ P^e_\Phi,\; P^f_\Psi\circ P^f_\Phi)
\]
again preserves admissible trajectories and the identities
\eqref{eq:morph-F}--\eqref{eq:morph-f}. Associativity is inherited from ordinary
composition of maps.
\end{proof}

\begin{definition}[Thermodynamically admissible subcategory]
The symbol \(\ThermAdm\) denotes the full subcategory of \(\ThermCont\) whose objects
are thermodynamically consistent.
\end{definition}

In this sense, the category is not merely a wrapper around the continuum equations;
it is the minimal formal environment in which preservation of the second-law
structure can be stated as an invariant property of morphisms.

\begin{theorem}[Thermodynamic consistency is invariant under exact morphisms]
\label{thm:invariance}
Let \(\Phi:\mathsf{T}_1\to\mathsf{T}_2\) be an exact thermodynamic morphism. If
\(\mathsf{T}_1\in\ThermAdm\), then \(\mathsf{T}_2\) satisfies the free-energy balance
along every trajectory in the image of \(T_\Phi\).
\end{theorem}

\begin{proof}
Let \(u\) be a smooth admissible trajectory of \(\mathsf{T}_1\), and define
\(u^\Phi:=T_\Phi u\). Because \(\mathsf{T}_1\) belongs to \(\ThermAdm\),
\[
\frac{\dd}{\dd t}\F_1[u(t)]
=
-\Diss_1[u(t)] + \Bnd_1[u(t)] + \Src_1[u(t)].
\]
By \eqref{eq:morph-F}--\eqref{eq:morph-S},
\[
\frac{\dd}{\dd t}\F_2[u^\Phi(t)]
=
-\Diss_2[u^\Phi(t)] + \Bnd_1[u(t)] + \Src_2[u^\Phi(t)].
\]
Using \eqref{eq:morph-e}, \eqref{eq:morph-f}, and compatibility of the pairings,
\[
\Bnd_2[u^\Phi(t)]
=
\pair{e_2[u^\Phi(t)]}{f_2[u^\Phi(t)]}_{Y_2}
=
\pair{e_1[u(t)]}{f_1[u(t)]}_{Y_1}
=
\Bnd_1[u(t)].
\]
Substituting yields
\[
\frac{\dd}{\dd t}\F_2[u^\Phi(t)]
=
-\Diss_2[u^\Phi(t)] + \Bnd_2[u^\Phi(t)] + \Src_2[u^\Phi(t)].
\]
\end{proof}

\section{Restriction, local-to-global reconstruction, and interconnection}

This section sharpens the compositional core of the manuscript. The key point is that
thermodynamic consistency is not only stable under passing to a subdomain or under
coupling two open subsystems; it also admits a genuine local-to-global
interpretation. Local balances on a compatible family of subdomains reconstruct the
global balance exactly, because the internal interface exchanges cancel pairwise.
This is the first place where the paper's categorical thesis acquires a distinctly
geometric form.

\subsection{Admissible subdomains and local thermodynamic data}

Let \(\mathsf{Dom}(\Om)\) denote the category whose objects are Lipschitz subdomains
\(U\subseteq \Om\), and whose morphisms are inclusions \(V\hookrightarrow U\). The
category is a poset category under set inclusion.

For each \(U\in \mathsf{Dom}(\Om)\), write
\[
\Gamma_U^{\mathrm{ext}}:=\pa U\cap \pa\Om,
\qquad
\Gamma_U^{\mathrm{int}}:=\pa U\cap \Om.
\]
If \(u\) is a smooth admissible trajectory on \(\Om\), define the restricted local
thermodynamic data on \(U\) by
\begin{align}
\F_U[u(t)]
&:=
\int_U \psi(x,u(x,t))\,\dd x,\label{eq:F-U}\\
\Diss_U[u(t)]
&:=
\int_U \sigma[u(t)]\,\dd x,\label{eq:D-U}\\
\Src_U[u(t)]
&:=
\int_U \mu[u(t)]\cdot R[u(t)]\,\dd x,\label{eq:S-U}\\
\Bnd_U^{\mathrm{ext}}[u(t)]
&:=
-\int_{\Gamma_U^{\mathrm{ext}}} q[u(t)]\cdot n_U\,\dd S,\label{eq:B-U-ext}\\
\Bnd_U^{\mathrm{int}}[u(t)]
&:=
-\int_{\Gamma_U^{\mathrm{int}}} q[u(t)]\cdot n_U\,\dd S,\label{eq:B-U-int}
\end{align}
where \(q=q^{\mathrm{rev}}+q^{\mathrm{diss}}\) and \(n_U\) is the outward unit normal
on \(\pa U\). We also set
\begin{equation}\label{eq:B-U-total}
\Bnd_U[u]:=\Bnd_U^{\mathrm{ext}}[u]+\Bnd_U^{\mathrm{int}}[u].
\end{equation}

\begin{definition}[Category of local thermodynamic data]
Let \(\mathsf{Data}\) denote the category whose objects are tuples
\[
(\F,\Diss,\Src,\Bnd^{\mathrm{ext}},\Bnd^{\mathrm{int}})
\]
associated with a Lipschitz subdomain, and whose morphisms are the restriction maps
induced by admissible inclusions of subdomains.
\end{definition}

\begin{definition}[Local thermodynamic assignment]
The \emph{local thermodynamic assignment} associated with a global system
\(\mathsf{T}\in\ThermAdm\) is the functor
\[
\mathbb{T}_{\mathsf{T}}:\mathsf{Dom}(\Om)^{\mathrm{op}}\to \mathsf{Data},
\qquad
U\mapsto
\bigl(\F_U,\Diss_U,\Src_U,\Bnd_U^{\mathrm{ext}},\Bnd_U^{\mathrm{int}}\bigr).
\]
\end{definition}

\begin{proposition}[Contravariant locality structure]\label{prop:local-assignment}
The assignment \(\mathbb{T}_{\mathsf{T}}\) is contravariant with respect to
admissible inclusions. In particular, if \(V\subseteq U\subseteq \Om\), then the
thermodynamic data on \(V\) are obtained from those on \(U\) by restriction of the
underlying fields and traces.
\end{proposition}

\begin{proof}
All local quantities are defined by integration or trace over the chosen subdomain.
If \(V\subseteq U\), then the restricted field \(u|_V\) determines \(\F_V\),
\(\Diss_V\), \(\Src_V\), and the corresponding external and internal boundary terms
by the same formulas. Compatibility with identity maps and composition of inclusions
is immediate.
\end{proof}

\begin{remark}
Proposition~\ref{prop:local-assignment} is only a first structural layer. It says
that thermodynamic data localize contravariantly over subdomains, but it does not
yet impose a sheaf axiom. The gluing content appears only once interface
compatibility is taken into account.
\end{remark}

\subsection{Restriction to a single subdomain}

\begin{theorem}[Restriction theorem]\label{thm:restriction}
Let \(\mathsf{T}\in\ThermAdm\), let \(u\) be a smooth admissible trajectory of
\(\mathsf{T}\), and let \(U\in\mathsf{Dom}(\Om)\). Then
\begin{equation}\label{eq:balance-U}
\frac{\dd}{\dd t}\F_U[u(t)]
=
-\Diss_U[u(t)]
+\Bnd_U^{\mathrm{ext}}[u(t)]
+\Bnd_U^{\mathrm{int}}[u(t)]
+\Src_U[u(t)].
\end{equation}
\end{theorem}

\begin{proof}
Integrate the local thermodynamic balance
\[
\partial_t \psi(\cdot,u) + \diver q[u]
=
-\sigma[u] + \mu[u]\cdot R[u]
\]
over \(U\). The divergence term splits over \(\Gamma_U^{\mathrm{ext}}\) and
\(\Gamma_U^{\mathrm{int}}\), giving exactly the terms
\eqref{eq:B-U-ext}--\eqref{eq:B-U-int}. Rearrangement yields
\eqref{eq:balance-U}.
\end{proof}

\begin{remark}
Theorem~\ref{thm:restriction} shows that restriction does not break the
thermodynamic law. It merely separates boundary exchange into two visible pieces:
external exchange with the environment and internal exchange across the newly exposed
interface.
\end{remark}

\subsection{Compatible partitions and local-to-global reconstruction}

\begin{definition}[Compatible finite partition]\label{def:compatible-partition}
A finite family \(\{\Om_i\}_{i=1}^N\subseteq \mathsf{Dom}(\Om)\) is called a
\emph{compatible finite partition} of \(\Om\) if:
\begin{enumerate}[label=(\roman*),leftmargin=1.1cm]
    \item \(\Om=\bigcup_{i=1}^N \cl{\Om_i}\);
    \item the interiors of the \(\Om_i\) are pairwise disjoint;
    \item each \(\Om_i\) is Lipschitz;
    \item for each \(i\neq j\), the interface
    \[
    \Gamma_{ij}:=\pa\Om_i\cap \pa\Om_j\cap \Om
    \]
    is either empty or a Lipschitz hypersurface on which the outer normals satisfy
    \[
    n_j=-n_i
    \qquad \text{a.e. on }\Gamma_{ij}.
    \]
\end{enumerate}
\end{definition}

For such a partition, define the local data on each \(\Om_i\) by
\[
\F_i:=\F_{\Om_i},\qquad
\Diss_i:=\Diss_{\Om_i},\qquad
\Src_i:=\Src_{\Om_i},
\]
and write
\begin{align}
\Bnd_i^{\mathrm{ext}}[u]
&=
-\int_{\pa\Om_i\cap \pa\Om} q[u]\cdot n_i\,\dd S,\label{eq:Biext}\\
\Bnd_i^{\mathrm{int}}[u]
&=
-\sum_{j\neq i}\int_{\Gamma_{ij}} q[u]\cdot n_i\,\dd S.\label{eq:Biint}
\end{align}

\begin{lemma}[Pairwise interface cancellation]\label{lem:pairwise-cancel}
Let \(\{\Om_i\}_{i=1}^N\) be a compatible finite partition of \(\Om\), and let \(u\)
be a smooth admissible trajectory. Then, for each nonempty interface \(\Gamma_{ij}\),
\begin{equation}\label{eq:pairwise-cancel}
-\int_{\Gamma_{ij}} q[u]\cdot n_i\,\dd S
-
\int_{\Gamma_{ij}} q[u]\cdot n_j\,\dd S
=
0.
\end{equation}
Consequently,
\begin{equation}\label{eq:sum-Bint-zero}
\sum_{i=1}^N \Bnd_i^{\mathrm{int}}[u]=0.
\end{equation}
\end{lemma}

\begin{proof}
On \(\Gamma_{ij}\), the compatible partition assumption gives \(n_j=-n_i\). Hence
\[
q[u]\cdot n_j = -\,q[u]\cdot n_i
\qquad \text{a.e. on }\Gamma_{ij},
\]
which proves \eqref{eq:pairwise-cancel}. Summing over all interfaces yields
\eqref{eq:sum-Bint-zero}.
\end{proof}

\begin{theorem}[Local-to-global reconstruction]\label{thm:local-to-global}
Let \(\mathsf{T}\in\ThermAdm\), let \(u\) be a smooth admissible trajectory, and let
\(\{\Om_i\}_{i=1}^N\) be a compatible finite partition of \(\Om\). Then:
\begin{align}
\F[u] &= \sum_{i=1}^N \F_i[u],\label{eq:F-sum}\\
\Diss[u] &= \sum_{i=1}^N \Diss_i[u],\label{eq:D-sum}\\
\Src[u] &= \sum_{i=1}^N \Src_i[u],\label{eq:S-sum}\\
\Bnd[u] &= \sum_{i=1}^N \Bnd_i^{\mathrm{ext}}[u],\label{eq:B-sum-ext}
\end{align}
and summing the local balances \eqref{eq:balance-U} over \(i\) recovers the global
balance
\begin{equation}\label{eq:global-recovered}
\frac{\dd}{\dd t}\F[u(t)]
=
-\Diss[u(t)] + \Bnd[u(t)] + \Src[u(t)].
\end{equation}
\end{theorem}

\begin{proof}
The identities \eqref{eq:F-sum}--\eqref{eq:S-sum} follow from additivity of
integrals over a partition. Summing \eqref{eq:balance-U} over \(i\) gives
\[
\frac{\dd}{\dd t}\sum_{i=1}^N \F_i[u(t)]
=
-\sum_{i=1}^N \Diss_i[u(t)]
+\sum_{i=1}^N \Bnd_i^{\mathrm{ext}}[u(t)]
+\sum_{i=1}^N \Bnd_i^{\mathrm{int}}[u(t)]
+\sum_{i=1}^N \Src_i[u(t)].
\]
Use \eqref{eq:F-sum}--\eqref{eq:S-sum}, the interface cancellation
\eqref{eq:sum-Bint-zero}, and the fact that the sum of the external pieces is the
global boundary exchange \eqref{eq:B-sum-ext}. This yields
\eqref{eq:global-recovered}.
\end{proof}

\begin{remark}[Local-to-global principle]
Theorem~\ref{thm:local-to-global} is the clearest expression of the paper's central
claim. The second law is not merely stable under restriction; it is reconstructible
from compatible local balances. Thermodynamic consistency therefore behaves as a
local-to-global invariant.
\end{remark}

\subsection{Open subsystems and power-conserving gluing}

\begin{definition}[Open local subsystem]
An \emph{open local subsystem} on a Lipschitz domain \(U\subseteq \Om\) is an object
\[
\mathsf{T}_U
=
(U,\X_U,\F_U,\Diss_U,\Src_U,
Y_{e,U}^{\Gamma},Y_{f,U}^{\Gamma},e_U^\Gamma,f_U^\Gamma)
\]
equipped with a local balance
\begin{equation}\label{eq:local-open-balance}
\frac{\dd}{\dd t}\F_U[u_U(t)]
=
-\Diss_U[u_U(t)]
+\Bnd_U^{\mathrm{ext}}[u_U(t)]
+\Bnd_U^{\Gamma}[u_U(t)]
+\Src_U[u_U(t)],
\end{equation}
where
\[
\Bnd_U^{\Gamma}[u_U]:=
\pair{e_U^\Gamma[u_U]}{f_U^\Gamma[u_U]}_{Y_U^\Gamma}.
\]
\end{definition}

\begin{definition}[Interface interconnection relation]
Let \(\mathsf{T}_{\Om_i}\) and \(\mathsf{T}_{\Om_j}\) be open local subsystems sharing an
interface \(\Gamma_{ij}\). An \emph{interface interconnection relation} is a subset
\[
\mathfrak{I}_{ij}
\subset
Y_{e,\Om_i}^{\Gamma_{ij}}\times Y_{f,\Om_i}^{\Gamma_{ij}}
\times
Y_{e,\Om_j}^{\Gamma_{ij}}\times Y_{f,\Om_j}^{\Gamma_{ij}}
\]
such that every quadruple
\((\eta_i,\zeta_i,\eta_j,\zeta_j)\in \mathfrak{I}_{ij}\) satisfies
\begin{equation}\label{eq:power-ij}
\pair{\eta_i}{\zeta_i}_{Y_{\Om_i}^{\Gamma_{ij}}}
+
\pair{\eta_j}{\zeta_j}_{Y_{\Om_j}^{\Gamma_{ij}}}
=0.
\end{equation}
\end{definition}

\begin{theorem}[Interconnection theorem]\label{thm:interconnection}
Let \(\{\Om_i\}_{i=1}^N\) be a compatible finite partition of \(\Om\), and let each
\(\Om_i\) carry an open local subsystem \(\mathsf{T}_{\Om_i}\) satisfying
\eqref{eq:local-open-balance}. Assume that on every nonempty interface \(\Gamma_{ij}\)
the subsystem trajectories satisfy a power-conserving interconnection relation
\(\mathfrak{I}_{ij}\). Then the composite system with total storage, dissipation,
source, and external boundary exchange
\begin{align}
\F_\#[u_1,\dots,u_N] &:= \sum_{i=1}^N \F_i[u_i],\label{eq:Fsharp-local}\\
\Diss_\#[u_1,\dots,u_N] &:= \sum_{i=1}^N \Diss_i[u_i],\label{eq:Dsharp-local}\\
\Src_\#[u_1,\dots,u_N] &:= \sum_{i=1}^N \Src_i[u_i],\label{eq:Ssharp-local}\\
\Bnd_\#[u_1,\dots,u_N] &:= \sum_{i=1}^N \Bnd_i^{\mathrm{ext}}[u_i]\label{eq:Bsharp-local}
\end{align}
satisfies
\begin{equation}\label{eq:sharp-local-balance}
\frac{\dd}{\dd t}\F_\#[u_1(t),\dots,u_N(t)]
=
-\Diss_\#[u_1(t),\dots,u_N(t)]
+\Bnd_\#[u_1(t),\dots,u_N(t)]
+\Src_\#[u_1(t),\dots,u_N(t)].
\end{equation}
\end{theorem}

\begin{proof}
Sum the local open balances \eqref{eq:local-open-balance} over \(i\). The total
interface contribution is
\[
\sum_{i=1}^N \Bnd_i^{\Gamma}[u_i(t)].
\]
Group these terms by interfaces \(\Gamma_{ij}\). By the power-conserving relation
\eqref{eq:power-ij}, each pair of interface contributions cancels. The remaining
terms are exactly the sums \eqref{eq:Fsharp-local}--\eqref{eq:Bsharp-local}, which
yields \eqref{eq:sharp-local-balance}.
\end{proof}

\begin{remark}
Theorem~\ref{thm:interconnection} is the converse companion to
Theorem~\ref{thm:local-to-global}. The latter says that a global admissible system
decomposes into local balances whose interface powers cancel. The former says that
compatible open local subsystems can be glued into a composite admissible system
provided the interface relations are power-conserving.
\end{remark}

\subsection{Presheaf viewpoint and future sheaf enhancement}

\begin{remark}[Presheaf candidate]
The contravariant assignment \(U\mapsto \mathbb{T}_{\mathsf{T}}(U)\) of
Proposition~\ref{prop:local-assignment} is naturally viewed as a presheaf-like
organization of thermodynamic data over the domain category \(\mathsf{Dom}(\Om)\).
The local-to-global reconstruction theorem shows that the free-energy balance is
compatible with passage from local pieces to the whole. What remains for a genuine
sheaf formulation is to identify the correct notion of matching family and the
corresponding gluing axiom for states, fluxes, and interface variables.
\end{remark}

\begin{remark}[What is proved here and what is deferred]
The present section stops short of asserting that thermodynamically admissible
systems form a sheaf on \(\mathsf{Dom}(\Om)\). That stronger statement would require a
careful treatment of admissible matching conditions on overlaps, uniqueness of gluing,
and the interaction of constitutive data with interfaces. What has been proved here
is the minimal continuum thermodynamic package on which such a sheaf-theoretic
extension could later be built: localization by restriction, cancellation of internal
power, and exact reconstruction of the global second-law balance from compatible local
balances.
\end{remark}

\begin{remark}[Categorical significance]
Theorems~\ref{thm:restriction}, \ref{thm:local-to-global}, and
\ref{thm:interconnection} are the first genuinely nontrivial expressions of the
paper's thesis. Restriction may be viewed as a localization operation, the
local-to-global theorem as exact reconstruction from compatible local balances, and
interconnection as a gluing operation subject to power-conserving interface data. The
second law is preserved under all three, so thermodynamic consistency behaves as a
local-to-global invariant.
\end{remark}

\section{Linearization as a derived theory}

We now show that classical linear convection--diffusion models arise as tangent
thermodynamic systems.

\begin{definition}[Reference equilibrium]
Let \(\mathsf{T}\in\ThermAdm\). A state \(u_\star\in\X\) is a
\emph{dissipation-neutral reference equilibrium} if
\begin{align}
\diver(J^{\mathrm{rev}}[u_\star]+J^{\mathrm{diss}}[u_\star])&=R[u_\star],
\label{eq:eq1}\\
J^{\mathrm{diss}}[u_\star]&=0,\qquad \sigma[u_\star]=0,\label{eq:eq2}\\
e[u_\star]&=0,\qquad f[u_\star]=0,\label{eq:eq3}\\
R[u_\star]&=0.\label{eq:eq4}
\end{align}
\end{definition}

\begin{assumption}[Linearizable constitutive class]\label{ass:lin}
Let \(u_\star\) be a reference equilibrium. Assume:
\begin{enumerate}[label=(\alph*),leftmargin=1.1cm]
    \item \(\psi\) is \(C^3\) in \(u\) near \(u_\star\);
    \item \(J^{\mathrm{rev}},R,e,f\) are Fr\'echet differentiable at \(u_\star\);
    \item the reference equilibrium is a state of vanishing dissipative driving force,
    \begin{equation}\label{eq:gradmu-star}
    \grad \mu[u_\star]=0
    \qquad \text{in }\Om;
    \end{equation}
    \item the dissipative flux has the form
    \begin{equation}\label{eq:gradient-diss}
    J^{\mathrm{diss}}[u]=-M(u)\grad \mu[u],
    \end{equation}
    where \(M(u)\) is pointwise symmetric positive semidefinite;
    \item there exists a linear flux \(q_{\lin}^{\mathrm{rev}}\) such that for every
    smooth perturbation \(w\),
    \begin{equation}\label{eq:lin-rev}
    \eta[w]\cdot \diver A_\star[w]
    =
    \diver q_{\lin}^{\mathrm{rev}}[w]
    \qquad \text{in }\mathscr{D}'(\Om),
    \end{equation}
    where the symbols below are well-defined.
\end{enumerate}
\end{assumption}

\begin{definition}[Linearized thermodynamic quantities]
Define the Hessian field
\begin{equation}\label{eq:Hstar}
H_\star(x):=D_{uu}^2\psi(x,u_\star(x)).
\end{equation}
For a perturbation \(w\), define
\begin{equation}\label{eq:eta}
\eta[w]:=H_\star w.
\end{equation}
The linearized operators are
\begin{align}
A_\star[w]&:=DJ^{\mathrm{rev}}[u_\star]w,\label{eq:Astar}\\
G_\star[w]&:=DR[u_\star]w,\label{eq:Gstar}\\
E_\star[w]&:=De[u_\star]w,\label{eq:Estar}\\
F_\star[w]&:=Df[u_\star]w,\label{eq:Fstar}
\end{align}
and the frozen mobility is \(M_\star:=M(u_\star)\). Under
\eqref{eq:gradmu-star} the linearized dissipative flux is
\begin{equation}\label{eq:Jlin}
J_{\lin}^{\mathrm{diss}}[w]
:=
-\,M_\star \grad \eta[w].
\end{equation}
\end{definition}

\begin{assumption}[Coercive equilibrium Hessian]\label{ass:hessian-coercive}
There exist constants \(0<\alpha\le \beta<\infty\) such that for almost every
\(x\in\Om\) and every \(\xi\in\R^m\),
\[
\alpha |\xi|^2 \le \xi^\top H_\star(x)\xi \le \beta |\xi|^2.
\]
\end{assumption}

\begin{definition}[Derived linearized system]
The derived linearized system at \(u_\star\) is
\begin{equation}\label{eq:lin-pde}
\partial_t w + \diver\bigl(A_\star[w]+J_{\lin}^{\mathrm{diss}}[w]\bigr)
=
G_\star[w].
\end{equation}
Its induced functionals are
\begin{align}
\F_{\lin}[w]
&:=
\frac12\int_\Om w^\top H_\star w\,\dd x,\label{eq:F-lin}\\
\Diss_{\lin}[w]
&:=
\int_\Om \grad \eta[w]:M_\star\grad \eta[w]\,\dd x,\label{eq:D-lin}\\
\Src_{\lin}[w]
&:=
\int_\Om \eta[w]\cdot G_\star[w]\,\dd x,\label{eq:S-lin}\\
\Bnd_{\lin}[w]
&:=
\pair{E_\star[w]}{F_\star[w]}_Y.\label{eq:B-lin}
\end{align}
\end{definition}

\begin{theorem}[Linearized free-energy balance]\label{thm:lin-balance}
Assume \eqref{eq:lin-rev}, and assume the boundary matching condition
\begin{equation}\label{eq:lin-boundary}
-\int_{\pa\Om}\bigl(q_{\lin}^{\mathrm{rev}}[w]+q_{\lin}^{\mathrm{diss}}[w]\bigr)\cdot n\,\dd S
=
\Bnd_{\lin}[w],
\end{equation}
where \(q_{\lin}^{\mathrm{diss}}[w]:=\eta[w]\cdot J_{\lin}^{\mathrm{diss}}[w]\).
Then every smooth solution \(w\) of \eqref{eq:lin-pde} satisfies
\begin{equation}\label{eq:lin-balance}
\frac{\dd}{\dd t}\F_{\lin}[w(t)]
=
-\Diss_{\lin}[w(t)] + \Bnd_{\lin}[w(t)] + \Src_{\lin}[w(t)].
\end{equation}
\end{theorem}

\begin{proof}
Differentiate \eqref{eq:F-lin}:
\[
\frac{\dd}{\dd t}\F_{\lin}[w(t)]
=
\int_\Om \eta[w]\cdot \partial_t w\,\dd x.
\]
Using \eqref{eq:lin-pde},
\[
\partial_t w
=
-\diver A_\star[w]
-\diver J_{\lin}^{\mathrm{diss}}[w]
+G_\star[w].
\]
Hence
\[
\frac{\dd}{\dd t}\F_{\lin}[w(t)]
=
-\int_\Om \eta[w]\cdot \diver A_\star[w]\\ dd x
-\int_\Om \eta[w]\cdot \diver J_{\lin}^{\mathrm{diss}}[w]\\ dd x
+\int_\Om \eta[w]\cdot G_\star[w]\\ dd x.
\]
By \eqref{eq:lin-rev},
\[
\eta[w]\cdot \diver A_\star[w]=\diver q_{\lin}^{\mathrm{rev}}[w].
\]
For the dissipative term,
\[
\diver q_{\lin}^{\mathrm{diss}}[w]
=
\grad \eta[w]:J_{\lin}^{\mathrm{diss}}[w]
+
\eta[w]\cdot \diver J_{\lin}^{\mathrm{diss}}[w]
=
-\grad \eta[w]:M_\star \grad \eta[w]
+
\eta[w]\cdot \diver J_{\lin}^{\mathrm{diss}}[w].
\]
Thus
\[
\eta[w]\cdot \diver J_{\lin}^{\mathrm{diss}}[w]
=
\diver q_{\lin}^{\mathrm{diss}}[w]
+
\grad \eta[w]:M_\star \grad \eta[w].
\]
Substituting these identities and integrating by parts yields
\[
\frac{\dd}{\dd t}\F_{\lin}[w(t)]
=
-\int_{\pa\Om}\bigl(q_{\lin}^{\mathrm{rev}}[w]+q_{\lin}^{\mathrm{diss}}[w]\bigr)\cdot n\,\dd S
-\int_\Om \grad \eta[w]:M_\star \grad \eta[w]\,\dd x
+\int_\Om \eta[w]\cdot G_\star[w]\,\dd x.
\]
Using \eqref{eq:lin-boundary} and \eqref{eq:D-lin}--\eqref{eq:S-lin} gives
\eqref{eq:lin-balance}.
\end{proof}

\begin{corollary}
Under Assumption~\ref{ass:hessian-coercive}, the quadratic storage \(\F_{\lin}\) is
positive definite and equivalent to the \(L^2\)-norm of \(w\). In particular, if
\(\Bnd_{\lin}=0\) and \(\Src_{\lin}=0\), then \(\F_{\lin}\) is a Lyapunov functional
for the closed linearized dynamics.
\end{corollary}

\begin{proof}
The coercivity estimate follows directly from
Assumption~\ref{ass:hessian-coercive}. The Lyapunov claim follows from
\eqref{eq:lin-balance}.
\end{proof}

The role of the categorical language here is modest but precise: linearization is not
merely a formal Taylor expansion of the PDE, but a passage that transports the
thermodynamic structure itself to the tangent level. What is inherited is not only a
linear evolution equation, but also a quadratic storage, a dissipative form, a source
term, and a boundary exchange structure satisfying the induced balance law.

\begin{definition}[Pointed admissible category]
Let \(\ThermAdm_\star\) denote the category whose objects are pairs
\((\mathsf{T},u_\star)\), where \(\mathsf{T}\in\ThermAdm\) and \(u_\star\) is a
reference equilibrium satisfying Assumptions~\ref{ass:lin} and
\ref{ass:hessian-coercive}. A morphism
\[
\Phi:(\mathsf{T}_1,u_{\star,1})\to(\mathsf{T}_2,u_{\star,2})
\]
is an exact thermodynamic morphism such that:
\begin{enumerate}[label=(\roman*),leftmargin=1.1cm]
    \item \(T_\Phi\) is Fr\'echet-smooth in a neighborhood of \(u_{\star,1}\);
    \item \(T_\Phi u_{\star,1}=u_{\star,2}\);
    \item the exact preservation identities for storage, dissipation, source, and
    boundary pairings hold for admissible states in a neighborhood of the reference
    equilibrium.
\end{enumerate}
\end{definition}

\begin{proposition}[Linearization functor]
Under the preceding assumptions, linearization defines a functor
\[
\mathbb{L}:\ThermAdm_\star\to \ThermLin,
\]
where \(\ThermLin\) is the category of linear thermodynamic systems equipped with a
balance of the form \eqref{eq:lin-balance}.
\end{proposition}

\begin{proof}
On objects, the assignment is the construction above. On morphisms, the Fr\'echet
differential \(DT_\Phi[u_{\star,1}]\) is well-defined by smoothness of \(T_\Phi\) near
the reference equilibrium. Because the preservation identities hold in a neighborhood
of the equilibrium, differentiation transports the nonlinear storage, dissipation,
source, and boundary pairings to their tangent counterparts. Identity morphisms map
to identity differentials, and composition is preserved by the chain rule.
\end{proof}

\begin{remark}
The functor \(\mathbb{L}\) formalizes the slogan that linear thermodynamics is the
tangent category of nonlinear continuum thermodynamics at equilibrium. The linear
theory is therefore best understood as a tangent thermodynamic category associated
with the nonlinear continuum theory at equilibrium.
\end{remark}

\section{Flagship continuum examples}

\subsection{Nonlinear drift--diffusion with advection}

Let \(u:\Om\times(0,T)\to (0,\infty)\) be scalar, let \(h\in C^2((0,\infty))\) with
\(h''>0\), let \(M:(0,\infty)\to [0,\infty)\) be \(C^1\), and let
\(b\in C^1(\cl{\Om};\R^d)\) satisfy
\begin{equation}\label{eq:b-divfree}
\diver b=0.
\end{equation}
Consider
\begin{equation}\label{eq:dd-pde}
\partial_t u + \diver\bigl(ub - M(u)\grad \mu\bigr)=r(u),
\qquad
\mu=h'(u).
\end{equation}
The storage density is
\[
\psi(u)=h(u),
\qquad
\F_{\mathrm{dd}}[u]=\int_\Om h(u)\,\dd x.
\]
Set
\[
J_{\mathrm{dd}}^{\mathrm{rev}}[u]=ub,
\qquad
J_{\mathrm{dd}}^{\mathrm{diss}}[u]=-M(u)\grad \mu.
\]

\begin{proposition}[Thermodynamic balance for drift--diffusion]
Assume \(u\) is a smooth positive solution of \eqref{eq:dd-pde}. Then
\begin{equation}\label{eq:dd-balance}
\frac{\dd}{\dd t}\F_{\mathrm{dd}}[u(t)]
=
-\int_\Om M(u)|\grad \mu|^2\,\dd x
+\Bnd_{\mathrm{dd}}[u(t)]
+\int_\Om \mu\,r(u)\,\dd x,
\end{equation}
where
\begin{equation}\label{eq:dd-bnd}
\Bnd_{\mathrm{dd}}[u]
=
-\int_{\pa\Om}\bigl(h(u)b-\mu M(u)\grad \mu\bigr)\cdot n\,\dd S.
\end{equation}
\end{proposition}

\begin{proof}
Since \(\diver b=0\),
\[
\mu\,\diver(ub)=h'(u)b\cdot \grad u = b\cdot \grad h(u)=\diver(h(u)b).
\]
Thus \(q^{\mathrm{rev}}[u]=h(u)b\). For the dissipative part,
\[
\mu\,\diver(-M(u)\grad \mu)
=
\diver(-\mu M(u)\grad \mu) + M(u)|\grad \mu|^2.
\]
Hence \(q^{\mathrm{diss}}[u]=-\mu M(u)\grad \mu\) and
\(\sigma[u]=M(u)|\grad \mu|^2\). The claim follows from
Theorem~\ref{thm:continuum-balance}.
\end{proof}

Restriction to subdomains follows directly from Theorem~\ref{thm:restriction}. The
interface interpretation requires a small extra hypothesis.

\begin{remark}[Interface ports for drift--diffusion]
Let \(\Om=\Om_1\cup\Om_2\) with common interface
\[
\Gamma=\pa\Om_1\cap \pa\Om_2\cap\Om.
\]
Assume additionally that the reversible advection is tangential to the coupling
interface, so that
\[
b\cdot n_i = 0
\qquad \text{on }\Gamma,\quad i=1,2.
\]
Then the interface boundary power is purely dissipative and may be written as
\[
\Bnd_i^\Gamma[u_i]
=
\int_\Gamma \mu_i\,\bigl(-J_{\mathrm{dd}}^{\mathrm{diss}}[u_i]\cdot n_i\bigr)\,\dd S.
\]
Accordingly, a natural interface effort/flow choice is
\[
e_i^\Gamma[u_i] := \mu_i|_\Gamma,
\qquad
f_i^\Gamma[u_i] := -J_{\mathrm{dd}}^{\mathrm{diss}}[u_i]\cdot n_i.
\]
The canonical lossless interconnection
\[
e_1^\Gamma=e_2^\Gamma,
\qquad
f_1^\Gamma=-f_2^\Gamma
\]
is then power-conserving, and the corresponding abstract gluing is a special case of
Theorem~\ref{thm:interconnection}. If advective transport across \(\Gamma\) is
allowed, the reversible contribution must be treated as an additional interface power
term and cannot be absorbed into this same effort/flow pair without modification.
\end{remark}

If \(u_\star>0\) is a constant equilibrium with \(r(u_\star)=0\), then
\[
H_\star=h''(u_\star),\qquad \eta=h''(u_\star)w,\qquad M_\star=M(u_\star),
\]
and the derived tangent system is
\[
\partial_t w + \diver\bigl(wb - M_\star \grad \eta\bigr)=r'(u_\star)w.
\]

\subsection{Porous-medium convection--diffusion}

Let \(m>1\), \(b\in C^1(\cl{\Om};\R^d)\) with \(\diver b=0\), and consider
\begin{equation}\label{eq:pm-pde}
\partial_t u + \diver\bigl(ub - \grad(u^m)\bigr)=s(u),
\qquad u\ge 0.
\end{equation}
Define
\[
\psi(u)=\frac{u^m}{m-1},
\qquad
\mu=\psi'(u)=\frac{m}{m-1}u^{m-1}.
\]
Then
\[
\grad(u^m)=\frac{m-1}{m}u\grad \mu.
\]
Hence
\[
J_{\mathrm{pm}}^{\mathrm{rev}}[u]=ub,
\qquad
J_{\mathrm{pm}}^{\mathrm{diss}}[u]=-\frac{m-1}{m}u\grad \mu.
\]

\begin{proposition}[Thermodynamic balance for porous-medium transport]
Assume \(u\) is a smooth nonnegative solution of \eqref{eq:pm-pde}. Then
\begin{equation}\label{eq:pm-balance}
\frac{\dd}{\dd t}\int_\Om \frac{u^m}{m-1}\,\dd x
=
-\int_\Om \frac{m-1}{m}u|\grad \mu|^2\,\dd x
+\Bnd_{\mathrm{pm}}[u]
+\int_\Om \mu s(u)\,\dd x,
\end{equation}
where
\[
\Bnd_{\mathrm{pm}}[u]
=
-\int_{\pa\Om}
\left(
\frac{u^m}{m-1}b-\mu \frac{m-1}{m}u\grad \mu
\right)\cdot n\,\dd S.
\]
\end{proposition}

\begin{proof}
The reversible identity is again
\[
\mu\,\diver(ub)=\diver(\psi(u)b)
\]
because \(\diver b=0\) and \(\psi'(u)=\mu\). The dissipative identity is
\[
\mu\,\diver\!\left(-\frac{m-1}{m}u\grad \mu\right)
=
\diver\!\left(-\mu\frac{m-1}{m}u\grad \mu\right)
+\frac{m-1}{m}u|\grad \mu|^2.
\]
Thus \(q^{\mathrm{rev}}[u]=\psi(u)b\),
\(q^{\mathrm{diss}}[u]=-\mu\frac{m-1}{m}u\grad \mu\), and
\(\sigma[u]=\frac{m-1}{m}u|\grad \mu|^2\). The result follows from
Theorem~\ref{thm:continuum-balance}.
\end{proof}

If \(u_\star>0\) is a positive constant equilibrium with \(s(u_\star)=0\), then
\[
H_\star=mu_\star^{m-2},
\qquad
\eta=mu_\star^{m-2}w,
\qquad
M_\star=\frac{m-1}{m}u_\star,
\]
and the tangent system becomes
\[
\partial_t w + \diver\bigl(wb-(m-1)u_\star^{m-1}\grad w\bigr)=s'(u_\star)w.
\]

The point of these examples is not merely that they fit the framework, but that they
show the framework was built at the right level of generality: nonlinear continuum
thermodynamics first, linear theory second.

\section{Weak formulation and structure-preserving discretization}

\subsection{Weak formulation}

Let \(V\hookrightarrow L^2(\Om;\R^m)\) be a reflexive Banach space. A weak form of
\eqref{eq:master-pde} is
\begin{equation}\label{eq:weak-form}
\pair{\partial_t u}{\varphi}_{V^\ast,V}
+
a^{\mathrm{rev}}(u;\varphi)
+
a^{\mathrm{diss}}(u;\varphi)
=
\ell_R(u;\varphi)+\ell_\partial(u;\varphi),
\end{equation}
for all \(\varphi\in V\). The key structural condition is that the thermodynamic test
\(\varphi=\mu[u]\) be admissible.

\begin{definition}[Thermodynamically admissible weak formulation]
The weak formulation \eqref{eq:weak-form} is \emph{thermodynamically admissible} if
testing with \(\varphi=\mu[u]\) yields
\begin{equation}\label{eq:weak-balance}
\frac{\dd}{\dd t}\F[u(t)]
=
-\Diss[u(t)] + \Bnd[u(t)] + \Src[u(t)]
\end{equation}
in the sense of distributions in time.
\end{definition}

\subsection{Semidiscretization}

Let \(V_h\subset V\) be finite-dimensional, and let \(\Pi_h:V\to V_h\) be a bounded
projection.

\begin{definition}[Semidiscrete thermodynamic scheme]
A semidiscrete scheme is a trajectory \(u_h:[0,T]\to V_h\) satisfying
\begin{equation}\label{eq:semi-scheme}
(\partial_t u_h,\varphi_h)_{L^2}
+
a_h^{\mathrm{rev}}(u_h;\varphi_h)
+
a_h^{\mathrm{diss}}(u_h;\varphi_h)
=
\ell_{R,h}(u_h;\varphi_h)+\ell_{\partial,h}(u_h;\varphi_h)
\end{equation}
for all \(\varphi_h\in V_h\), together with the discrete thermodynamic potential
\[
\mu_h:=\Pi_h \mu[u_h]\in V_h.
\]
\end{definition}

\begin{definition}[Thermodynamic semidiscrete admissibility]
The semidiscrete scheme is \emph{thermodynamically admissible} if
\begin{align}
(\partial_t u_h,\mu_h)_{L^2}
&=
\frac{\dd}{\dd t}\F_h[u_h],\label{eq:semi-chain}\\
a_h^{\mathrm{rev}}(u_h;\mu_h)
&=
-\Bnd_h^{\mathrm{rev}}[u_h],\label{eq:semi-rev}\\
a_h^{\mathrm{diss}}(u_h;\mu_h)
&=
\Diss_h[u_h]-\Bnd_h^{\mathrm{diss}}[u_h],
\qquad \Diss_h[u_h]\ge 0,\label{eq:semi-diss}\\
\ell_{R,h}(u_h;\mu_h)&=\Src_h[u_h],\label{eq:semi-source}\\
\ell_{\partial,h}(u_h;\mu_h)&=\Bnd_h^{\mathrm{ext}}[u_h].\label{eq:semi-bound}
\end{align}
Setting
\[
\Bnd_h:=\Bnd_h^{\mathrm{rev}}+\Bnd_h^{\mathrm{diss}}+\Bnd_h^{\mathrm{ext}},
\]
we call \(\mathsf{T}_h=(V_h,\F_h,\Diss_h,\Bnd_h,\Src_h)\) a semidiscrete
thermodynamic object.
\end{definition}

\begin{theorem}[Abstract semidiscrete thermodynamic balance]\label{thm:semi-balance}
Every thermodynamically admissible semidiscrete scheme satisfies
\begin{equation}\label{eq:semi-balance}
\frac{\dd}{\dd t}\F_h[u_h(t)]
=
-\Diss_h[u_h(t)] + \Bnd_h[u_h(t)] + \Src_h[u_h(t)].
\end{equation}
\end{theorem}

\begin{proof}
Choose \(\varphi_h=\mu_h\) in \eqref{eq:semi-scheme} and use
\eqref{eq:semi-chain}--\eqref{eq:semi-bound}. Rearrangement yields
\eqref{eq:semi-balance}.
\end{proof}

\subsection{Fully discrete time stepping}

Let \(0=t^0<t^1<\cdots<t^N=T\), \(\tau_n=t^{n+1}-t^n\), and \(u_h^n\in V_h\).

\begin{definition}[One-step thermodynamic potential]
A \emph{one-step thermodynamic potential} is a map assigning to
\((u_h^n,u_h^{n+1})\) an element
\[
\overline{\mu}_h^{\,n+\frac12}\in V_h
\]
such that
\begin{equation}\label{eq:fd-chain}
\bigl(\overline{\mu}_h^{\,n+\frac12},u_h^{n+1}-u_h^n\bigr)_{L^2}
=
\F_h[u_h^{n+1}]-\F_h[u_h^n].
\end{equation}
\end{definition}

\begin{definition}[Thermodynamically admissible one-step scheme]
A one-step scheme is
\begin{equation}\label{eq:fd-scheme}
\begin{aligned}
\left(\frac{u_h^{n+1}-u_h^n}{\tau_n},\varphi_h\right)_{L^2}
&+
a_{h,\tau}^{\mathrm{rev}}(u_h^n,u_h^{n+1};\varphi_h)\\
&+
a_{h,\tau}^{\mathrm{diss}}(u_h^n,u_h^{n+1};\varphi_h)\\
&=
\ell_{R,h,\tau}(u_h^n,u_h^{n+1};\varphi_h)
+
\ell_{\partial,h,\tau}(u_h^n,u_h^{n+1};\varphi_h)
\end{aligned}
\end{equation}
for all \(\varphi_h\in V_h\), and is \emph{thermodynamically admissible} if testing
with \(\overline{\mu}_h^{\,n+\frac12}\) yields
\begin{align}
a_{h,\tau}^{\mathrm{rev}}(u_h^n,u_h^{n+1};\overline{\mu}_h^{\,n+\frac12})
&=
-\Bnd_{h,\tau}^{\mathrm{rev},\,n+\frac12},\label{eq:fd-rev}\\
a_{h,\tau}^{\mathrm{diss}}(u_h^n,u_h^{n+1};\overline{\mu}_h^{\,n+\frac12})
&=
\Diss_{h,\tau}^{\,n+\frac12}
-\Bnd_{h,\tau}^{\mathrm{diss},\,n+\frac12},
\qquad
\Diss_{h,\tau}^{\,n+\frac12}\ge 0,\label{eq:fd-diss}\\
\ell_{R,h,\tau}(u_h^n,u_h^{n+1};\overline{\mu}_h^{\,n+\frac12})
&=
\Src_{h,\tau}^{\,n+\frac12},\label{eq:fd-source}\\
\ell_{\partial,h,\tau}(u_h^n,u_h^{n+1};\overline{\mu}_h^{\,n+\frac12})
&=
\Bnd_{h,\tau}^{\mathrm{ext},\,n+\frac12}.\label{eq:fd-bound}
\end{align}
\end{definition}

\begin{theorem}[Abstract fully discrete thermodynamic balance]\label{thm:fully-discrete}
Every thermodynamically admissible one-step scheme satisfies
\begin{equation}\label{eq:fd-balance}
\F_h[u_h^{n+1}] - \F_h[u_h^n]
=
-\tau_n \Diss_{h,\tau}^{\,n+\frac12}
+
\tau_n \Bnd_{h,\tau}^{\,n+\frac12}
+
\tau_n \Src_{h,\tau}^{\,n+\frac12},
\end{equation}
where
\[
\Bnd_{h,\tau}^{\,n+\frac12}
:=
\Bnd_{h,\tau}^{\mathrm{rev},\,n+\frac12}
+\Bnd_{h,\tau}^{\mathrm{diss},\,n+\frac12}
+\Bnd_{h,\tau}^{\mathrm{ext},\,n+\frac12}.
\]
\end{theorem}

\begin{proof}
Choose \(\varphi_h=\overline{\mu}_h^{\,n+\frac12}\) in \eqref{eq:fd-scheme}, multiply
by \(\tau_n\), and use \eqref{eq:fd-chain} and
\eqref{eq:fd-rev}--\eqref{eq:fd-bound}. Rearrangement yields
\eqref{eq:fd-balance}.
\end{proof}

\begin{remark}
If the discrete chain rule is replaced by a one-step inequality, as in convex
splitting, the same proof yields a fully discrete thermodynamic inequality in the
correct dissipative direction.
\end{remark}

\begin{remark}
From the present viewpoint, discretization is not external to the theory but a
descendant construction that should preserve the thermodynamic invariant in an
appropriate discrete category. A natural continuation is to organize continuum,
semidiscrete, and fully discrete thermodynamic models within a hierarchy of
invariant-preserving categories linked by discretization functors.
\end{remark}

\section{Discussion and conclusions}

The manuscript was built around a simple but consequential reversal of viewpoint.
Instead of starting from linear irreversible thermodynamics and then extending it
outward, we started from nonlinear continuum convection--diffusion systems as
primitive open thermodynamic objects. This change in starting point changes the
relevant invariant. The central structural object is not a linear transport operator
or a special symmetry coefficient, but the free-energy balance
\[
\frac{\dd}{\dd t}\F = -\Diss + \Bnd + \Src,
\qquad
\Diss\ge 0.
\]
That balance is the second law in open-system form, and it is the property we asked
to survive under the operations actually used in continuum modeling.

The resulting picture is layered but coherent.
\begin{itemize}[leftmargin=1.2cm]
    \item At the continuum level, constitutive admissibility implies the free-energy
    balance.
    \item At the categorical level, exact morphisms preserve the balance.
    \item At the compositional level, restriction exposes internal transport as
    interface exchange, local compatible pieces reconstruct the global balance, and
    interconnection cancels internal interface power.
    \item At the tangent level, linear convection--diffusion theory is recovered as
    the linearized thermodynamic descendant of the nonlinear theory.
    \item At the computational level, semidiscrete and fully discrete schemes inherit
    the same structure when they are built around thermodynamic test variables and
    one-step chain rules.
\end{itemize}

This is why the categorical language is genuinely useful here. A single PDE can be
studied without it, but the moment one asks whether cutting, gluing, transforming,
linearizing, and discretizing preserve the second-law structure, one is no longer
studying an isolated equation. One is studying a class of objects and
invariant-preserving operations between them.

\paragraph{Further categorical directions.}
A natural continuation of the present framework is to formalize locality more
explicitly by organizing thermodynamic data over subdomains as a presheaf, and by
identifying compatibility conditions under which such data satisfy a sheaf-like
gluing principle. A further continuation is to investigate whether admissible
thermodynamic systems admit a useful internal semantics, possibly through geometric
logic or related topos-theoretic structures. These extensions are not required for
the present results, but the restriction--interconnection calculus developed here is
designed to make them plausible.

\paragraph{Conclusions.}
We have developed a continuum-first and nonlinear-first framework for nonlinear
convection--diffusion systems in which thermodynamic consistency is treated as a
categorical invariant. Open continuum systems were defined by state, storage,
constitutive, source, and port data; the free-energy balance was derived from local
constitutive admissibility; and preservation of that balance was proved under exact
morphisms, restriction, local-to-global reconstruction, interconnection,
linearization, semidiscretization, and fully discrete time stepping. The nonlinear
drift--diffusion and porous-medium examples show that the framework captures genuine
continuum PDEs rather than merely formal abstractions. The main conceptual outcome is
that the second law is best understood not as an auxiliary estimate attached to a
model after the fact, but as the organizing invariant of a category of admissible
nonlinear continuum transport systems and of their tangent, local, and computational
descendants.

\appendix

\section{Reversible neutrality and boundary ports}

This appendix records the concrete identities underlying the abstract assumptions on
reversible transport and boundary exchange.

The purpose of this appendix is to make precise the concrete transport and trace
identities that, in the main text, were treated at the level of abstract
thermodynamic structure.

\begin{proposition}[Advective reversible neutrality]\label{prop:advective}
Let \(\psi=\psi(u)\) have no explicit \(x\)-dependence, let
\(J^{\mathrm{rev}}[u]=u\otimes b\), and assume \(\diver b=0\). Then
\begin{equation}\label{eq:advective-neutral}
\mu[u]\cdot \diver J^{\mathrm{rev}}[u]
=
\diver\bigl(\psi(u)b\bigr)
\qquad \text{in }\mathscr{D}'(\Om).
\end{equation}
\end{proposition}

\begin{proof}
Since \(\diver b=0\),
\[
\diver(u\otimes b)=b\cdot \grad u.
\]
Therefore
\[
\mu[u]\cdot \diver J^{\mathrm{rev}}[u]
=
D_u\psi(u)\cdot (b\cdot \grad u)
=
b\cdot \grad \psi(u)
=
\diver(\psi(u)b).
\]
\end{proof}

\begin{remark}
This proposition is the source of the reversible flux choice \(q^{\mathrm{rev}}=\psi b\)
in the main examples. More general reversible fluxes can also be treated, but
\eqref{eq:advective-neutral} is the canonical convection identity.
\end{remark}

\begin{definition}[Boundary port realization]
Let \(\Gamma\subseteq \pa\Om\). A \emph{boundary port realization} on \(\Gamma\) is a
choice of trace spaces \(Y_e(\Gamma)\), \(Y_f(\Gamma)\), a duality pairing
\(\pair{\cdot}{\cdot}_\Gamma\), and trace maps
\[
e_\Gamma:\X\to Y_e(\Gamma),\qquad f_\Gamma:\X\to Y_f(\Gamma),
\]
such that
\[
\Bnd_\Gamma[u]=\pair{e_\Gamma[u]}{f_\Gamma[u]}_\Gamma.
\]
\end{definition}

\begin{example}[Potential--normal-flux pairing]
For a dissipative flux \(J^{\mathrm{diss}}[u]=-M(u)\grad \mu[u]\), a canonical choice is
\[
e_\Gamma[u]=\mu[u]|_\Gamma,
\qquad
f_\Gamma[u]=-J^{\mathrm{diss}}[u]\cdot n|_\Gamma.
\]
Then
\[
\Bnd_\Gamma^{\mathrm{diss}}[u]
=
\pair{\mu[u]|_\Gamma}{-J^{\mathrm{diss}}[u]\cdot n|_\Gamma}_\Gamma.
\]
\end{example}

\begin{proposition}[Power-conserving interface matching]
Let two subsystems share an interface \(\Gamma\) and let their port variables satisfy
\[
e_1^\Gamma=e_2^\Gamma,
\qquad
f_1^\Gamma=-f_2^\Gamma.
\]
Then the interface power cancels:
\[
\pair{e_1^\Gamma}{f_1^\Gamma}_\Gamma
+
\pair{e_2^\Gamma}{f_2^\Gamma}_\Gamma
=0.
\]
\end{proposition}

\begin{proof}
Immediate from bilinearity:
\[
\pair{e_1^\Gamma}{f_1^\Gamma}_\Gamma
+
\pair{e_2^\Gamma}{f_2^\Gamma}_\Gamma
=
\pair{e_1^\Gamma}{f_1^\Gamma}_\Gamma
+
\pair{e_1^\Gamma}{-f_1^\Gamma}_\Gamma
=0.
\]
\end{proof}

\begin{remark}
These trace-based port realizations are also the natural starting point for any later
sheaf-theoretic treatment of locality and interface gluing.
\end{remark}

\section{Discrete gradients and convex splitting}

\subsection{Exact discrete gradients}

Let \(X_h\) be a finite-dimensional inner-product space with inner product
\((\cdot,\cdot)_h\), and \(\F_h:X_h\to\R\) be \(C^1\).

\begin{definition}[Discrete gradient]
A map
\[
\overline{\nabla}\F_h:X_h\times X_h\to X_h
\]
is a \emph{discrete gradient} if for all \(u,v\in X_h\),
\begin{align}
\bigl(\overline{\nabla}\F_h(u,v),v-u\bigr)_h
&=
\F_h(v)-\F_h(u),\label{eq:dg1}\\
\overline{\nabla}\F_h(u,u)
&=
\nabla \F_h(u).\label{eq:dg2}
\end{align}
\end{definition}

\begin{proposition}[Average-vector-field discrete gradient]
The map
\[
\overline{\nabla}\F_h(u,v)
:=
\int_0^1 \nabla \F_h\bigl((1-\theta)u+\theta v\bigr)\,\dd\theta
\]
is a discrete gradient.
\end{proposition}

\begin{proof}
Let \(\gamma(\theta)=(1-\theta)u+\theta v\). Then
\[
\F_h(v)-\F_h(u)
=
\int_0^1 \frac{\dd}{\dd\theta}\F_h(\gamma(\theta))\,\dd\theta
=
\int_0^1 (\nabla \F_h(\gamma(\theta)),v-u)_h\,\dd\theta
=
\bigl(\overline{\nabla}\F_h(u,v),v-u\bigr)_h.
\]
Consistency on the diagonal is immediate.
\end{proof}

\subsection{Convex splitting}

\begin{definition}[Convex splitting]
A decomposition
\[
\F_h=\F_h^{\mathrm{c}}-\F_h^{\mathrm{e}}
\]
is a \emph{convex splitting} if both \(\F_h^{\mathrm{c}}\) and
\(\F_h^{\mathrm{e}}\) are convex. The associated one-step potential is
\[
\overline{\mu}_h^{\,n+\frac12}
=
\nabla \F_h^{\mathrm{c}}(u_h^{n+1})
-
\nabla \F_h^{\mathrm{e}}(u_h^n).
\]
\end{definition}

\begin{proposition}[Convex-splitting chain inequality]
For a convex splitting,
\[
\bigl(\overline{\mu}_h^{\,n+\frac12},u_h^{n+1}-u_h^n\bigr)_h
\ge
\F_h[u_h^{n+1}] - \F_h[u_h^n].
\]
\end{proposition}

\begin{proof}
By convexity of \(\F_h^{\mathrm{c}}\),
\[
\F_h^{\mathrm{c}}[u_h^{n+1}] - \F_h^{\mathrm{c}}[u_h^n]
\le
\bigl(\nabla \F_h^{\mathrm{c}}(u_h^{n+1}),u_h^{n+1}-u_h^n\bigr)_h.
\]
By convexity of \(\F_h^{\mathrm{e}}\),
\[
\F_h^{\mathrm{e}}[u_h^{n+1}] - \F_h^{\mathrm{e}}[u_h^n]
\ge
\bigl(\nabla \F_h^{\mathrm{e}}(u_h^n),u_h^{n+1}-u_h^n\bigr)_h.
\]
Subtracting the second inequality from the first yields the claim.
\end{proof}

\begin{remark}
The exact discrete-gradient construction yields one-step balance identities; convex
splitting yields one-step free-energy inequalities. Both are thermodynamically
admissible descendants of the semidiscrete theory.
\end{remark}

\begin{remark}
Thus the fully discrete descendants inherit the same storage--dissipation bookkeeping
in a form compatible with exact or inequality-based thermodynamic admissibility.
\end{remark}

\end{document}